\documentclass{amsart}

\usepackage{amsmath}
\usepackage[all]{xy}  %for \xymatrix
\usepackage{amssymb}  %for \trianglelefteq, \mathbb etc.
\usepackage{amsthm}   %for proof, thm, cor, etc.
\usepackage{amscd}      %for Commutative Diagrams

\providecommand{\abs}[1]{\left\lvert#1\right\rvert}
\providecommand{\Spec}[1]{\text{Spec }#1}

\DeclareMathOperator{\Hom}{Hom}
\DeclareMathOperator{\PGL}{PGL}
\DeclareMathOperator{\GL}{GL}

\def\N{\mathbb{N}}
\def\Z{\mathbb{Z}}
\def\Q{\mathbb{Q}}
\def\P{\mathbb{P}}
\def\A{\mathbb{A}}
\def\F{\mathbb{F}}
\def\O{\mathcal{O}}
\def\m{\mathfrak{m}}

\newcommand{\tth}{^{\operatorname{th}}}
\newcommand{\smooth}{\operatorname{smooth}}

\theoremstyle{plain}% default
\newtheorem{thm}{Theorem}
\newtheorem*{thm*}{Theorem}
\newtheorem{lem}[thm]{Lemma}
\newtheorem{prop}[thm]{Proposition}
\newtheorem{cor}[thm]{Corollary}
\theoremstyle{definition}
\newtheorem{defn}[thm]{Definition}

\theoremstyle{remark}
\newtheorem*{rem}{Remark}

\begin{document}
    \title{Good Reduction of Periodic Points on Projective Varieties}
    \author[Hutz]{Benjamin Hutz}
    \address{Department of Mathematics and Computer Science,
            Amherst College,
            Amherst, MA}
    \email{bhutz@amherst.edu}

    \thanks{The author thanks his advisor Joseph Silverman for his many insightful suggestions and ideas, and also Dan Abramovich and Robert Benedetto for their suggestions.}

\keywords{periodic points, good reduction, dynamical systems}

\subjclass[2000]{
11G99,
  % Arithmetic algebraic geometry (Diophantine geometry): None of the above
14G99
  % Arithmetic problems. Diophantine geometry: None of the above
(primary);
37F99
  % Complex dynamical systems: None of the above.
(secondary)}

%\date{\today}

\maketitle

    \begin{abstract}
      We consider the dynamical system created by iterating a morphism of a projective variety defined over the field of fractions of a discrete valuation ring.  We study the primitive period of a periodic point in this field in relation to the primitive period of the reduced point in the residue field, the order of the action on the cotangent space, and the characteristic of the residue field.
    \end{abstract}

\section{Introduction}
    We consider dynamical systems arising from iterating a morphism of a projective variety defined over the field of fractions of a discrete valuation ring.  Our goal is to obtain information about the dynamical system over the field of fractions by studying the dynamical system over the residue field.  In particular, we aim to bound the possible primitive periods of a periodic point.  This topic is discussed for $\phi:\P^1 \to \P^1$ with many references in \cite[Section 2.6]{Silverman10}.

    Recall that given a set $A$ and map $f:A \to A$ we can create a dynamical system by iterating the map $f$ on the set $A$.  We denote $f^n$ as the $n\tth$ iterate of $f$.  An element $a\in A$ such that $f^n(a) = a$ for some positive integer $n$ is called a \emph{periodic point} and the least such $n$ is called the \emph{primitive period of $a$}.  We will use the following notation unless otherwise specified:
    \begin{itemize}
        \item $R$ is a discrete valuation ring complete with respect to a normalized valuation $v$.
        \item $\m$ is the maximal ideal of $R$ with uniformizer $\pi$.
        \item $K$ is the field of fractions of $R$.
        \item $k = R/\m$ is the finite residue field of characteristic $p$.
        \item $\overline{\phantom{P}}$ denotes reduction mod $\pi$.
        \item $\abs{\phantom{P}}_v$ is the associated non-archimedean absolute value.
    \end{itemize}
    Note that these conditions imply that $K$ is a local field.

    In Section \ref{sect2} we establish a notion of good reduction for a projective variety $X/K$ and a morphism $\phi:X \to X$ defined over $K$ so that we can study the dynamics of $\phi$ over $K$ by examining the dynamics of the reduced map over $k$.

    In Section \ref{sect3} we describe the primitive period of $P \in X(K)$ with the following two theorems.
    \begin{thm} \label{thm_intro_mrp}
        Let $\mathcal{X}/R \subseteq \P^{N}_R$ be a smooth projective model of
        $X/K$, a non-singular irreducible projective variety of dimension $d$.
        Let $\phi_R:\mathcal{X}/R \to \mathcal{X}/R$ be an $R$-morphism and
        $P_R \in \mathcal{X}(R)$ be a periodic point of primitive period $n$ for $\phi_R$ with $\overline{P}$ of primitive period $m$.
        Then
        \begin{equation*}
            n=m
        \end{equation*}
        or there exists a $\phi$-stable subspace $V$ of the cotangent space of $\mathcal{X} \times_R \Spec{k}$ such that
        \begin{equation*}
            n=mr_Vp^e \quad \text{with } e \geq 0,
        \end{equation*}
        where $r_V$ is the order of the map induced by $\overline{\phi^m}$ on $V$.
        Furthermore, let $\mathcal{V} \subset \mathcal{X}$ be the scheme theoretic union of $\{P_R,\phi_R(P_R),\ldots,\phi^{n-1}_R(P_R)\}$.  Then $V$ is the cotangent space of $\mathcal{V} \times_R \Spec{k}$.  Let $d^{\prime} \leq d$ be the dimension of $V$, then
        \begin{equation*}
            r_V \leq (N\pi)^{d^{\prime}} -1
        \end{equation*}
        where $N\pi$ is the norm of $\pi$.
    \end{thm}
    This theorem generalizes the known result for $\phi:\P^{1} \to \P^1$ \cite[Theorem 2.21]{Silverman10} and is similar to \cite{Fakhruddin2}.

    \begin{thm} \label{thm_intro_pe}
        Let $\mathcal{X}/R \subseteq \P^{N}_R$ be a smooth projective model of
        $X/K$, a non-singular irreducible projective variety of dimension $d$.
        Let $\phi_R:\mathcal{X}/R \to \mathcal{X}/R$ be an $R$-morphism and
        $P_R \in \mathcal{X}(R)$ be a periodic point of primitive period $n$ for $\phi_R$.
        Using the notation from Theorem \ref{thm_intro_mrp}, we have for $n=mr_Vp^e$ that
        \begin{equation*}
            e \leq
            \begin{cases}
              1 + \log_2(v(p)) & p \neq 2\\
              1 + \log_{\alpha}\left(\frac{\sqrt{5}v(2) + \sqrt{5(v(2))^2 + 4}}{2}\right) & p=2.
            \end{cases}
        \end{equation*}
        where $\alpha = \frac{1+\sqrt{5}}{2}$.
    \end{thm}
    This theorem is a generalization of results such as those found in \cite{Li,Pezda2,Pezda,Zieve} and, in particular, implies that the primitive period of a $K$-rational periodic point is bounded.  In the following corollary we state what the bound is when working over $\Q$.
    \begin{cor} \label{cor_intro}
        Let $X/\Q$ be a smooth irreducible projective variety of dimension $d$ and $\phi:X \to X$ a morphism defined over $\Q$ with good reduction at a prime $p$.  Let $P \in X(\Q)$ be a periodic point with primitive period $n$.  Then we have
        \begin{equation*}
            e \leq
                \begin{cases}
                  1 & p \neq 2\\
                  3 & p =2
                \end{cases}
        \end{equation*}
        and
        \begin{equation*}
            n \leq \begin{cases}
              p^{d+1}(p^d-1)  & p \neq 2 \\
              2^{d+3}(2^d-1)  & p=2.
            \end{cases}
        \end{equation*}
    \end{cor}
    By assuming that $\phi:X \to X$ is \'etale we may remove the hypotheses of smooth and irreducible to obtain the following theorem.
    \begin{thm} \label{thm_etale}
        Let $\phi:X \to X$ be an \'etale morphism of a projective variety defined over $K$.  Let $P\in X(K)$ be a smooth periodic point for $\phi$ with primitive period $n$.  Let $Y \subseteq X_{\smooth}$ be the irreducible component containing $P$.  Let $l$ be the smallest integer such that $\phi^l$ is an \'etale endomorphism of $Y$.  Assume that $\phi^l$ restricted to $Y$ has good reduction.  Then we have $n=lm$, $n=lmr_V$, or $n=lmr_Vp^e$ where $m$, $r_V$, and $p^e$ are as in Theorem \ref{thm_intro_mrp} and Theorem \ref{thm_intro_pe}.
    \end{thm}
    It is well-known that for $X \subset \P^N$ a variety defined by homogeneous polynomials of degree at most $d$, the number of irreducible components of $X$ is bounded by $d^N$.  In particular, Theorem \ref{thm_etale} implies that the primitive period of a smooth $K$-rational periodic point is bounded for an \'etale morphism of a projective variety.

\section{Good Reduction}\label{sect2}
    In this section, we consider the more general situation of a scheme $X/\Spec{K}$ and a $K$-morphism
    $\phi:X \to X$, unless otherwise stated.  Following \cite[\S 2.5]{Silverman10} we define a notion of good reduction.
    \begin{defn} \label{defn2}
        A scheme $X/K$ has \emph{good reduction} if there exists a smooth proper scheme $\mathcal{X}/R$ with generic fiber $X/K$.  We call such an $\mathcal{X}/R$ a \emph{smooth proper model} for $X/K$.
    \end{defn}
    If $X/K$ has good reduction, each point in $X(K)$ corresponds to a unique point in the proper scheme $\mathcal{X}(R)$ and, consequently, a unique point in the special fiber (denoted as $\overline{P}$).   In addition to a notion of good reduction for a scheme $X/K$, we also need a notion of
    good reduction for a $K$-morphism $\phi:X \to X$.
    \begin{defn} \label{defn1}
        Let $X/K$ be a scheme and $\phi:X \to X$ a $K$-morphism.  We say that $\phi$ has \emph{good reduction} if there exists a smooth proper model $\mathcal{X}/R$ of $X/K$ and an $R$-morphism $\phi_R: \mathcal{X} \to \mathcal{X}$ extending $\phi$.  Denote the restriction of $\phi_R$ to the special fiber as $\overline{\phi}$.
    \end{defn}

    \begin{rem}
        Let $X = \P^1_K$ and $\phi:\P^1_K \to \P^1_K$ be a morphism defined over $K$.
        Then good reduction as defined in Definition \ref{defn1}
        is equivalent to good reduction as defined in \cite[Theorem 2.15]{Silverman10}.  For $X = \P^{N}_K$ and $\phi:X \to X$, a morphism over $K$, we can again formulate a definition of good reduction using resultants that is equivalent to Definition \ref{defn1}; see for example \cite[Section 1.1]{Silverman11}.
    \end{rem}

    \begin{rem}
        For a morphism $\phi:\P^{1}_K \to \P^{1}_K$, Hsia \cite[Section 3]{Hsia} defines a notion of \emph{mildly bad reduction} as the case where $\phi$ has bad reduction but there exists a projective scheme $\mathcal{X}$ of finite type over $R$ and an $R$-morphism from $\mathcal{X} \to \P^{1}_R$ that is an isomorphism on the generic fiber such that $\phi$ extends to an $R$-morphism that maps the smooth
        part of $\mathcal{X}$ to itself.
    \end{rem}
    The following theorem and corollary show that for morphisms with good reduction, the dynamics of $\phi$
    are related to the dynamics of $\overline{\phi}$.
    \begin{thm}  \label{thm1}
        Let $\mathcal{X}/R$ be a smooth proper scheme with generic fiber $X/K$.  Let
        $\phi_R:\mathcal{X}/R \to \mathcal{X}/R$ be an $R$-morphism and let $\phi:X/K \to X/K$ be the
        restriction of $\phi_R$ to the generic fiber.  Let $\overline{\phi}$ be the restriction of $\phi_R$ to the
        special fiber.
        \begin{enumerate}
            \item \label{item5} $\overline{\phi}(\overline{P}) = \overline{\phi(P)}$, for all $P \in
                X(K)$.
            \item \label{item6} Let $\psi_R:\mathcal{X}/R \to \mathcal{X}/R$ be another $R$-morphism and $\psi:X/K \to X/K$ be the restriction of $\psi_R$ to the generic fiber.  Then
                $\overline{\phi \circ \psi} = \overline{\phi} \circ \overline{\psi}$.
            \item \label{item7} Let $\phi_R^n : \mathcal{X}/R \to \mathcal{X}/R$ be the $n\tth$ iterate of $\phi_R$
                and let $\overline{\phi^n}$ be $\phi_R^n$ restricted to the generic fiber.  Then
                \begin{equation*}
                    \overline{\phi^n(P)} = \overline{\phi}^n(\overline{P}) \quad \text{for all } P \in X(K).
                \end{equation*}
        \end{enumerate}
    \end{thm}

    \begin{proof}
        \mbox{}
        \begin{enumerate}
            \item Let $P \in X(K)$, i.e., $P \in \Hom_{\Spec K}(\Spec K, X)$.  Since the scheme is proper, we have a unique associated
                \begin{equation*}
                    P_R \in \Hom_{\Spec R}(\Spec R, \mathcal{X}).
                \end{equation*}
                Using the universal property of fiber products, it is easy to see that
                \begin{equation*}
                    (\phi_R \circ P_R)_{|_{\Spec k}} = (\phi_R)_{|_{\Spec k}} \circ
                    (P_R)_{|_{\Spec k}}.
                \end{equation*}

            \item A composition of morphisms is a morphism, so we have that $\phi_R \circ \psi_R$
                is also a morphism of schemes.  Using an argument with fiber products,
                we see that $\overline{\phi \circ \psi} = \overline{\phi} \circ \overline{\psi}$.

            \item To prove this statement we proceed by induction on $n$ by applying (\ref{item6}) to the maps $\phi_R$ and $\phi_R^{n-1}$.
        \end{enumerate}
    \end{proof}

    \begin{rem}
        There are some interesting questions to be raised concerning good reduction of $K$-morphisms.
        \begin{itemize}
            \item If $\phi^2$ has good reduction, does that necessarily imply $\phi$ has good reduction?
                For $N=1$, Benedetto \cite[Theorem B]{Benedetto} proves for $\phi \in K(x)$, a
                rational map of degree $\geq 2$, and any positive integer $n$ that $\phi$
                has good reduction if and only if $\phi^n$ has good reduction.

            \item If $\phi$ and $\psi$ are distinct and both have good reduction,
                then is it necessarily true
                that $\phi \circ \psi$ has good reduction?  The complication is whether a smooth proper model $\mathcal{X}/R$ exists where both
                $\phi$ and $\psi$ extend to $R$-morphisms $\phi_R$ and $\psi_R$.  It is not clear if
                the good reduction of $\phi$ and $\psi$ is enough to ensure the existence of
                such a smooth proper model.
        \end{itemize}
    \end{rem}

    \begin{defn}
        Let $X/K$ be a scheme, $\phi:X \to X$ be a $K$-morphism, and $P \in X(K)$.
        \begin{itemize}
            \item  The point $P$ is \emph{periodic} if $\phi^n(P) =P$ for some $n \in \N$.  The integer $n$ is called a \emph{period} of $P$.
            \item  If $P$ is periodic with period $n$ and $\phi^{m}(P) \neq P$ for all $0 < m < n$, then $n$ is called the \emph{primitive period} of $P$.
            \item  If $\phi^{m+n}(P) = \phi^{m}(P)$ for some $m,n \in \N$, then $P$ is called \emph{preperiodic}.
        \end{itemize}
    \end{defn}

    \begin{cor} \label{cor1}
        Let $\mathcal{X}/R$ be a smooth proper scheme with generic fiber $X/K$.  Let
        $\phi_R:\mathcal{X}/R \to \mathcal{X}/R$ be an $R$-morphism and $\phi:X/K \to X/K$ the
        restriction of $\phi_R$ to the generic fiber.  Let $\overline{\phi}$ be the restriction of $\phi_R$ to the  special fiber.  Then the reduction map sends periodic points to periodic
        points and preperiodic points to preperiodic points.
        Furthermore, if $P \in X(K)$ has primitive period $n$ and $\overline{P}$ has
        primitive period $m$, then $m \mid n$.
    \end{cor}

    \begin{proof}
        Let $P \in X(K)$ be a point of primitive period $n$ and let $m$ be the primitive period of $\overline{P}$.
        We are given that $\phi$ has good reduction, so by Theorem \ref{thm1} we know
        \begin{equation}\label{eq5}
            \overline{\phi^{n}(P)} = \overline{\phi}^{n}(\overline{P}) = \overline{P}.
        \end{equation}
        From (\ref{eq5}) we deduce that $m \leq n$ and $n \equiv 0 \mod m$, since the smallest period $m$ is the
        greatest common divisor of all of the periods.
    \end{proof}
    Before we examine the primitive period, we state what it means for a scheme defined over a number field to have good
    reduction at a particular prime.
    \begin{defn}
        Let $L$ be a number field with ring of integers $A$ and $X/L$ a scheme over $L$.
        Let $\mathfrak{p}$ be a prime of $L$ and $A_{\mathfrak{p}}$ the localization of
        $A$ at $\mathfrak{p}$.  We say that $X$ has \emph{good reduction at
        $\mathfrak{p}$} if there exists a smooth proper model
        $\mathcal{X}/A_{\mathfrak{p}}$ with generic fiber $X/L$.

        Similarly, $\phi:X/L \to X/L$ has \emph{good reduction at $\mathfrak{p}$} if $\phi$
        has good reduction over $A_\mathfrak{p}$.
    \end{defn}

    For the rest of this article we work with morphisms of projective
    varieties defined over $K$.  We assume that $\mathcal{X}/R$ is, in fact, a smooth
    \emph{projective} scheme with generic fiber $X/K$.  Since projective implies proper,
    this is slightly more restrictive.

\section{Description of the Primitive Period}\label{sect3}
\subsection{Preliminary Results}
    \begin{defn}
        Let $P \in \P^N_K$ be a point.  We call a representation of $P$ as $[P_0,\ldots,P_N]$ with $P_i \in R$ for $0 \leq i \leq N$ and at least one $P_i \in R^{\ast}$ a \emph{normalization of $P$}.
        We define the \emph{reduction of $P$ modulo $\pi$}, denoted $\overline{P}$, by
        first choosing a normalization of $P$ and then setting
        \begin{equation*}
            \overline{P} = [\overline{P_0},\ldots,\overline{P_N}] \in \P^{N}_k.
        \end{equation*}
        Note that $\overline{P}$ is independent of the choice of normalization.
    \end{defn}
    We recall some standard facts about projective space.
    \begin{prop} \label{prop3}
        Let $q$ be the number of elements of $k$. Then
        \begin{enumerate}
            \item \label{item8} Given any $N+2$ points $P_i \in \P^{N}_R$ for which the images $P_i(\Spec{K})$ and their reductions (the images $P_i(\Spec{k})$) satisfy that no $N+1$ of them are co-planar, there exists a transformation in $\PGL_{N+1}(R)$ that maps them to any other $N+2$ points $Q_i$ in $\P^{N}_R$ for which the images $Q_i(\Spec{K})$ and their reductions satisfy that no $N+1$ of them are co-planar.
            \item \label{item10} The number of hyperplanes of $\P^{N}_k$ is $q^N + q^{N-1}
                + \cdots + q + 1$.
            \item \label{item11} The number of hyperplanes of $\P^{N}_k$ through a
                point of $\P^{N}_k$ is $q^{N-1} + \cdots +q+1$.
        \end{enumerate}
    \end{prop}
    \begin{proof}
        \mbox{}
        \begin{enumerate}
            \item It is a standard argument to show that given $N+2$ points in $\mathbb{P}^{N+1}_K$ with no $N+1$ of them co-planar, that we can find a element of $\PGL_{N+1}(K)$ that takes them to any other set of $N+2$ points in $\mathbb{P}^{N+1}_K$ with no $N+1$ of them co-planar.  Using the same argument and the fact that their reductions (points in $\mathbb{P}^{N+1}(k)$) also satisfy that no $N+1$ of them are co-planar.  We can find an element in $\PGL_{N+1}(R)$ that takes them to any other set of $N+2$ points in $\mathbb{P}^{N+1}_R$ with no $N+1$ of them co-planar.

            \item A $d$ dimensional subspace of $\mathbb{P}^{N}_k$ is isomorphic to
            $\mathbb{P}^{d}_k$ which has $q^d + q^{d-1} + \cdots + q + 1$ points.

            \item Let $P \in \mathbb{P}^{N}_k$ and $H$ a hyperplane of $\mathbb{P}^{N}_k$
                not containing $P$.  Any hyperplane of $\mathbb{P}^{N}_k$ through $P$ meets
                $H$ in a hyperplane of $H$.  By (\ref{item10}), there are $q^{N-1} + \cdots + q + 1$
                hyperplanes of $H$.
        \end{enumerate}
    \end{proof}
    For the proof of Theorem \ref{thm_intro_mrp}, we need a moving lemma for the orbit of a point
    whose reduction is a fixed point.  Let $\A^{N}_i$ be the standard affine open sets in $\P^{N}$ obtained by sending
    \begin{equation*}
        [t_0,\ldots, t_N] \to \left(\frac{t_0}{t_i}, \ldots,
        \frac{t_{i-1}}{t_i},\frac{t_{i+1}}{t_i},\ldots, \frac{t_{N}}{t_i}\right).
    \end{equation*}

    \begin{lem} \label{lem1}
        Let $N \geq 2$.
        Given any finite set of points $\mathcal{P}=\{P_\alpha\}_{\alpha \in I}
        \subset \P^{N}_R$ whose image in the special fiber is a single point
        and any fixed $i$, $0 \leq i \leq N$, we can find a transformation $f \in \PGL_{N+1}(R)$
        such that $f(\mathcal{P}) \subset (\A^{N}_i)_R$.
    \end{lem}

    \begin{proof}
        Let $q$ be the number of elements of the residue field $k$.

        Without loss of generality, consider $(\A^{N}_0)_R$.
        We need an $f \in \PGL_{N+1}(R)$ that sends the points of $\mathcal{P}$ to points not in the hyperplane $t_0=0$.  So we need to find a hyperplane $H$ in $\P^{N}_R$ that does not contain any of the points of $\mathcal{P}$ and send it to the hyperplane $t_0=0$.

        The image in the special fiber of $\mathcal{P}$ is a single point, denoted $\overline{\mathcal{P}}$.
        By Proposition \ref{prop3} (\ref{item10}) and (\ref{item11}), there are
        \begin{equation*}
            (q^{N}+\cdots + q + 1) - (q^{N-1} + \cdots + q + 1) = q^{N}
        \end{equation*}
        hyperplanes of $\P^{N}_k$  which do not go through
        $\overline{\mathcal{P}}$.  Let $\overline{H}$ be any hyperplane in $\P^{N}_k$ that does not contain
        $\overline{\mathcal{P}}$.  Choose any $N$ points on
        $\overline{H}$. We need to find two additional points in $\P^{N}_k$
        which, when combined with the already chosen $N$ points, satisfy that no $N+1$ of them
        are co-planar in $\P^{N}_k$.  Let one such point be $\overline{\mathcal{P}}$, which is not
        on $\overline{H}$ by construction.

        There are $\binom{N+1}{N}=N+1$ subsets of $N$ points of the $N+1$ chosen points.  We need to
        choose an additional point which does not lie on a hyperplane containing any of
        those $N+1$ subsets.  Each subset defines a unique hyperplane, so we must choose a point not on $N+1$
        hyperplanes.  By Proposition \ref{prop3} (\ref{item10}), there are
        \begin{equation*}
            q^N + \cdots + q + 1
        \end{equation*}
        total hyperplanes.  So we need
        \begin{equation*}
            q^N + \cdots + q + 1 > N+1
        \end{equation*}
        to be able to find such a point.  We know that $q\geq 2$ and $N \geq 2$, so we can
        always find such a point.

        Having found the necessary points on the special fiber, we can lift them to (not necessary unique) points
        on the generic fiber.  Since $\P^{N}_R$ is proper, these points on the generic fiber
        correspond to unique points of the scheme.  We have found $N+2$ points that
        satisfy the hypothesis of Proposition \ref{prop3} (\ref{item8}); hence, there exists an
        element $f$ of $\PGL_{N+1}(R)$ as desired.
    \end{proof}

    To be able to apply Lemma \ref{lem1}, we need to be certain that we do not change the dynamics.
    \begin{defn}
        Let $X \subset \P^{N}_K$ be a projective variety, $\phi:X \to X$
        a morphism, and $f \in \PGL_{N+1}(R)$.  Define $\phi^f = f^{-1} \circ \phi \circ
        f$.
    \end{defn}

    \begin{prop} \label{prop30}
        Let $\mathcal{X}/R \subset \P^{N}_R$ be a smooth proper model for $X$,
        a projective variety defined over $K$.  Let $\phi_R:
        \mathcal{X} \to \mathcal{X}$ be an $R$-morphism extending $\phi:X/K \to X/K$.
        Let $f \in \PGL_{N+1}(R)$ and let $P \in X(K)$ be a periodic point of primitive
        period $n$ whose reduction modulo $\pi$ has primitive period $m$.  Then,
        \begin{enumerate}
            \item \label{item12} $\phi^f:f^{-1}(X) \to f^{-1}(X)$ has good reduction.
            \item \label{item13} $f^{-1}(P)$ has primitive period $n$ for $\phi^f$.
            \item \label{item14} $\overline{f^{-1}(P)}$ has primitive period $m$ for $\overline{\phi^f}$.
        \end{enumerate}
    \end{prop}

    \begin{proof}
        Note that the embedding $\mathcal{X}/R \subset \P^{N}_R$ induces an embedding
        $X/K \subset \P^{N}_K$, so the action of $f$ on $X$ is defined.
        We need to show that there exists a smooth projective model
        $\mathcal{Y}/R$ and an $R$-morphism $\psi_R: \mathcal{Y} \to \mathcal{Y}$ such that
        the following diagram commutes and the vertical maps are isomorphisms.
        \begin{equation}\label{eq43}
            \xymatrix{
                \mathcal{Y}/R \ar[r]^{\psi_R} \ar[d]_{f_R} & \mathcal{Y}/R \ar[d]^{f_R}  \\
                \mathcal{X}/R \ar[r]^{\phi_R} & \mathcal{X}/R
            }
        \end{equation}

        We know that $f \in \PGL_{N+1}(R)$ is an automorphism of $\P^{N}_K$ that has good
        reduction with smooth projective model $\P^{N}_R$.  Let $f_R:\P^{N}_R \to
        \P^{N}_R$ be the morphism extending $f$.  Then $f_R$ is an automorphism of
        $\P^{N}_R$ and
        \begin{equation*}
            f_R^{-1}:\mathcal{X}/R \to f_R^{-1}(\mathcal{X}/R)
        \end{equation*}
        is an isomorphism.  Let $\mathcal{Y}/R$ be defined as $f_R^{-1}(\mathcal{X}/R)$, thus the vertical maps are isomorphisms.  Let
        \begin{equation*}
            \psi_R = \phi_R^{f_R} = f^{-1}_R \circ \phi_R \circ f_R.
        \end{equation*}
        Diagram (\ref{eq43}) now clearly commutes.

        To prove (\ref{item12}), the smooth projective model is $\mathcal{Y}/R$ and the morphism is $\psi_R$.
        Since the vertical maps are isomorphisms and
        \begin{equation*}
            (f^{-1} \circ \phi \circ f)^s = f^{-1} \circ \phi^s \circ f \quad \text{for
            all } s \geq 1,
        \end{equation*}
        (\ref{item13}) and (\ref{item14}) follow immediately from the diagram (\ref{eq43}).
    \end{proof}

    To prove Theorem \ref{thm_intro_mrp}, we need a $\pi$-adic version of the Implicit Function
    Theorem that includes an explicit bound on the size of the neighborhood of convergence.
    To do this, we recall without proof a version of Hensel's Lemma for systems of power series.  This version of Hensel's Lemma comes from Greenberg \cite[Chapter 5]{Greenberg}, and we adopt his
    terminology.

    \begin{defn}
        We say $\textbf{f}$ is a \emph{system of formal power series} over $R$ if $\textbf{f}$ is a vector
        where each component is a formal power series with coefficients in $R$ having zero constant term.
    \end{defn}

    \begin{prop}[Hensel's Lemma] \label{thm2}
        Let
        \begin{equation*}
            F_1,\ldots,F_N \in R[x_1,\ldots,x_N] \quad \text{and} \quad
            a_1,\ldots,a_N \in R
        \end{equation*}
        with
        \begin{equation*}
            \abs{F_i(a_1,\ldots,a_N)}_v < \abs{\det\left(\left(\frac{\partial F_i}
            {\partial x_j}\right)_{i,j}(a_1,\ldots,a_N)\right)}_v^2 \leq 1
        \end{equation*}
        for $1 \leq i \leq N$.  Then there exist $b_1,\ldots,b_N \in R$ with
        \begin{equation*}
            F_i(b_1,\ldots,b_N)=0 \text{ and } \abs{b_i-a_i}_v <\abs{\det\left(\left(\frac{\partial F_i}
            {\partial x_j}\right)_{i,j}(a_1,\ldots,a_N)\right)}_v
        \end{equation*}
        for $1 \leq i \leq N$.  Furthermore, the system of power series taking $(a_1,\ldots,a_N)$ to the solution
        $(b_1,\ldots,b_N)$ is defined over $R$.
    \end{prop}

    Recall that a function $\textbf{F}:\A^{N+M}(K) \to \A^{M}(K)$
    is \emph{smooth} at a point $P \in \A^{N+M}(K)$ if its Jacobian matrix has rank $M$ at $P$.

    \begin{prop}[$\pi$-adic Implicit Function Theorem] \label{thm3}
        Let
        \begin{equation*}
            \textbf{F}=[F_1,\ldots,F_M]:\A^{N+M}(K) \to \A^{M}(K)
        \end{equation*}
        be a smooth function with each $F_i$ defined over $R$.
        Label the coordinates of $\A^{N+M}$ as $(x_1,\ldots,x_N,y_1,\ldots y_M)$.
        Let $(\textbf{a}_0,\textbf{b}_0)$ be a point in $\A^{N+M}(K)$ such that
        \begin{enumerate}
            \item $F(\textbf{a}_0,\textbf{b}_0)=0$,
            \item the minor
                \begin{equation*}
                    \mathop{\left(\frac{\partial F_i}{\partial y_j}\right)_{i=1\ldots M}}_{\hspace*{32pt} j=1\ldots M}
                    (\textbf{a}_0,\textbf{b}_0)
                \end{equation*}
                is full rank, and
            \item
                \begin{equation*}
                    \abs{F_i(\textbf{a}_0,\textbf{b}_0)}_v < \delta^2
                    \quad \text{for } 1 \leq i \leq M
                \end{equation*}
        \end{enumerate}
        with
        \begin{equation*}
            \delta =\abs{\det\left(\mathop{\left(\frac{\partial F_i}{\partial y_j}\right)_{i=1\ldots M}}_{\hspace*{32pt} j=1\ldots M}
            (\textbf{a}_0,\textbf{b}_0)\right)}_v \leq 1.
        \end{equation*}
        Let $B_{\epsilon}$ denote the $\pi$-adic open ball of radius $\epsilon$ around the origin.
        Then for all $\textbf{a}$ such that
        \begin{equation*}
            \abs{\textbf{a}-\textbf{a}_0}_v < \delta,
        \end{equation*}
        there exists a unique system over $R$
        \begin{equation*}
            \textbf{g}:B_{\delta^2} \to B_{\delta} \subset \A^{M}
        \end{equation*}
        such that
        \begin{equation*}
            \textbf{g}(\textbf{a}_0)=\textbf{b}_0 \quad \text{and} \quad \textbf{F}(\textbf{a},g(\textbf{a})) = 0.
        \end{equation*}
    \end{prop}

    \begin{proof}
        By assumption, we have that
        \begin{equation*}
            \abs{F_i(\textbf{a}_0,\textbf{b}_0)}_v < \delta^2
            \quad \text{for } 1 \leq i \leq M.
        \end{equation*}
        After a translation, we may assume $\textbf{a}_0=(0,\ldots,0)$.

        For any $\textbf{a}$ and $i$, $1 \leq i \leq M$, $F_i(\textbf{a},\textbf{b}_0)$ has a zero constant
        term in $x_1,\ldots,x_N$; so for
        \begin{equation*}
            \abs{a_i}_v < \delta^2 \quad \text{for } 1 \leq i \leq N
        \end{equation*}
         we have
        \begin{equation*}
            \abs{F_i(a_1,\ldots,a_N,\textbf{b}_0)}_v < \delta^2
            \quad \text{for } 1 \leq i \leq M.
        \end{equation*}
        We may apply Hensel's Lemma (Proposition \ref{thm2}), which concludes that there exists a
        system of power series $\textbf{g}$ defined over $R$ which takes a point
        $(a_1,\ldots,a_N)$ and sends it to the unique point
        $\textbf{g}(a_1,\ldots,a_N)=(b_1,\ldots,b_M)$, satisfying
        \begin{enumerate}
            \item $b_1,\ldots,b_M \in R$,
            \item $\textbf{F}(a_1,\ldots,a_N,b_1,\ldots,b_M)=0$, and
            \item $\abs{b_i-a_i}_v < \delta \quad \text{for } 1 \leq i \leq M$.
        \end{enumerate}
        Note that since $\abs{a_i}_v < \delta^2 \quad \text{for } 1 \leq i \leq N$ and
        $\abs{b_i-a_i}_v < \delta \quad \text{for } 1 \leq i \leq M$,
        we have
        \begin{equation*}
            \textbf{g}:B_{\delta^{2}} \to B_{\delta}.
        \end{equation*}
    \end{proof}

\subsection{Proof of Theorem \ref{thm_intro_mrp}}
    Let $\mathcal{X}/R$ be a smooth projective $R$-scheme of dimension $d$ whose
    generic fiber is a non-singular irreducible projective variety $X/K$.  Let
    $\phi_R:\mathcal{X}/R \to \mathcal{X}/R$ be an $R$-morphism and denote
    the restrictions to the generic fiber and the special fiber as $\phi$ and $\overline{\phi}$, respectively.
    Let $P_R \in \mathcal{X}(R)$ be a periodic point of primitive period $n$ for $\phi_R$ and let
    $P = P_R(\Spec K)$ and $\overline{P} = P_R(\Spec k)$ with $\overline{P}$ of primitive period $m$ for $\overline{\phi}$.  There are three main steps in analyzing the primitive period of the reduction of a
    periodic point.
    \begin{itemize}
        \item  Use good reduction to show that there is an open
            $\pi$-adic neighborhood $\mathcal{U}$ and a system of functions $\textbf{f}$ regular
            on $\mathcal{U}$ and defined over $K$ such that $\phi_R$ is represented by
            $\textbf{f}$ as a regular map on $\mathcal{U}$
            and all of the iterates of $P_R$ by $\phi_R$ are contained in $\mathcal{U}$.
        \item Show that $\textbf{f}$ has a local power series representation on
            $\mathcal{U}$ with coefficients in $R$.
        \item Noticing that $\phi$ acts as the cyclic group of order $n$ on the scheme theoretic union $\mathcal{V}$ of the finite set of points $\{P_R,\phi_R(P_R),\ldots,\phi^{n-1}_R(P_R)\}$, iterate a local power series representation of $\textbf{f}$ to obtain information about $n$.
    \end{itemize}
    Recall that we know $m \mid n$ from Corollary \ref{cor1}.  Replacing $\phi_R$ by $\phi_R^{m}$ and $n$ by
    $n/m$, we may assume that $m=1$.  We resolve the first two steps in the following lemma.
    \begin{lem} \label{lem17}
        Let $\mathcal{X}/R \subseteq \P^{N}_R$ be a smooth projective model of
        $X/K$ a non-singular irreducible projective
        variety of dimension $d$.
        Let $\phi_R:\mathcal{X}/R \to \mathcal{X}/R$ be an $R$-morphism and
        let $P_R \in \mathcal{X}(R)$ be a periodic point of primitive period $n$ for $\phi_R$ such that $\overline{P}$ is fixed by $\overline{\phi}$.  Let $\mathcal{V} \subset \mathcal{X}$ be the scheme theoretic union of $\{P_R,\phi_R(P_R),\ldots,\phi^{n-1}(P_R)\}$.
        Then
        \begin{enumerate}
            \item  there exists a $g \in \PGL_{N+1}(R)$ such that
                \begin{equation*}
                    g^{-1}( \{P_R,\phi_R(P_R),\ldots,\phi^{n-1}(P_R)\}) \subset (\A^{N}_0)_R
                \end{equation*}
                and
                \begin{equation*}
                     g^{-1}(P_R) = [1,0,\ldots,0].
                \end{equation*}
            \item there exists a $\pi$-adic neighborhood $\mathcal{U}$ and regular maps $\textbf{f}=[f_i]$ defined over $R$ on $\mathcal{U}$ such that
                \begin{enumerate}
                    \item $\phi_R(\rho)=\textbf{f}(\rho)$ for
                            all $\rho \in \mathcal{U}$ and all of the iterates of $P_R$ by $\phi_R$ are contained in $\mathcal{U}$.
                    \item each $\textbf{f}_i$ has a power series representation over $R$ on $\mathcal{U} \cap (\A^{N}_0)_R$ that converges for values in $\m$ and that we can iterate at $P_R$.
                    \item the local system of parameters for $\O_{\mathcal{V},P_R}$ is a subset of the local system of parameters for $\O_{\mathcal{X},P_R}$.  We can write $\phi_R |_{\mathcal{V}} = [f'_i]$ for regular maps $\textbf{f}\;'=[f'_i]$ defined over $R$, where each $f'_i$ has a local power series representation over $R$ on $\mathcal{V}$ in the local system of parameters for $\O_{\mathcal{V},P_R}$ that converges for values in $\m$ and that we can iterate at $P_R$.
                \end{enumerate}
        \end{enumerate}
    \end{lem}

    \begin{proof}
        The embedding $\mathcal{X} \subseteq \P^{N}_R$ induces an embedding
        $X \subset \P^{N}_K$.  Let $[t_0,\ldots,t_N]$ be the coordinates of
        $\P^{N}_R$.
        Let
        \begin{equation*}
            \mathcal{P} = \{P_R, \phi_R(P_R), \ldots, \phi^{n-1}_R(P_R)\}.
        \end{equation*}
        This is a finite set of points in $\mathcal{X}(R)$ whose images on the special fiber
        are a single point $\overline{P}$.
        By Lemma \ref{lem1}, we can find a $g \in \PGL_{N+1}(R)$ such that
        \begin{equation*}
            g^{-1}(\mathcal{P}) \subset (\A^{N}_0)_R \quad \text{and} \quad
            g^{-1}(P_R(\Spec{R})) = [1,0,\ldots,0].
        \end{equation*}
        Replace $\phi_R$ by $\phi_R^g = g^{-1} \circ \phi_R \circ g$ and
        $P_R$ by $g^{-1}(P_R)$.
        We are now working over $g^{-1}(\mathcal{X})$, so replace $\mathcal{X}$ by
        $g^{-1}(\mathcal{X})$ and hence $\mathcal{V}$ by $g^{-1}(\mathcal{V})$.  Note that $\mathcal{V} \subset (\A_0^{N})_R$.
        Dehomogenize with respect to $t_0$ and label the affine coordinates
        \begin{equation*}
            T_i = \frac{t_i}{t_0} \quad \text{for } 1 \leq i \leq N.
        \end{equation*}
        Let $\mathcal{Y}/R = (\mathcal{X}/R) \cap (\A^{N}_0)_R$.  Note that
        $\mathcal{Y}$ is a smooth irreducible affine scheme and, consequently, integral.
        We know that $\mathcal{V} \subset \mathcal{Y}$ is a closed subscheme.  In particular, we have that
        \begin{equation*}
            \m_{\mathcal{V},P_R}/\m_{\mathcal{V},P_R}^2 \subset \m_{\mathcal{Y},P_R}/\m_{\mathcal{Y},P_R}^2
        \end{equation*}
        as vector spaces.  Hence, we can choose a basis for $\m_{\mathcal{Y},P_R}/\m_{\mathcal{Y},P_R}^2$ that contains a basis for $\m_{\mathcal{V},P_R}/\m_{\mathcal{V},P_R}^2$ as a subset.  Since we are working in $(\A_0^N)_R$, this change of basis comes from an element of $GL_{N}(R)$ and, hence, an element of $\PGL_{N+1}(R)$ acting on $\mathcal{X}$.  Continue to label the coordinates as $\{T_1,\ldots,T_N\}$.
        Let $g_1,\ldots,g_s \in R[T_1,\ldots,T_N]$ be equations that define $\mathcal{Y}$,
        i.e.,
        \begin{equation*}
            \mathcal{Y} = \Spec{R[T_1,\ldots,T_N]/(g_1,\ldots,g_s)}.
        \end{equation*}
        Identifying the points $\mathcal{Y}(R) = \Hom(\Spec{R}, \mathcal{Y})$ with the solutions to
        $g_1=\cdots=g_s=0$ in $\A^{N}_R$, we have
        the resulting map and periodic point on $\mathcal{Y}$:
        \begin{equation*}
            \phi_R: \mathcal{Y} \to \mathcal{X} \quad \text{ with } \quad
                \phi_R:\mathcal{P}_R \to \mathcal{P}_R, \quad \text{and} \quad P_R =
                (0,\ldots,0).
        \end{equation*}
        Since the scheme $\mathcal{Y}$ is smooth, the determinants of the $(N-d) \times (N-d)$ minors of the Jacobian matrix
        \begin{equation*}
            \mathop{\left(\frac{\partial g_i}{\partial T_j}\right)_{i=1\ldots s}}_{\hspace*{38pt} j=1\ldots
            N}
        \end{equation*}
        generate the unit ideal in $R[T_1,\ldots,T_N]/(g_1,\ldots,g_s)$ (see for
        example \cite[IV.2]{silverman2}).  In particular, since $R$ is a discrete valuation ring,
        $P_R$ a non-singular point of $\mathcal{Y}$ implies that the determinant of one of
        the $(N-d) \times (N-d)$ minors of the matrix
        \begin{equation*}
            \mathop{\left(\frac{\partial g_i}{\partial T_j}(P_R)\right)_{i=1\ldots s}}_{\hspace*{60pt} j=1\ldots N}
        \end{equation*}
        is a unit in $R$.  Relabeling, let
        \begin{equation*}
            \mathop{\left(\frac{\partial g_i}{\partial T_j}(P_R)\right)_{i=1\ldots N-d}}_{\hspace*{56pt} j=d+1\ldots
            N}
        \end{equation*}
        be such a minor and, hence,
        \begin{equation*}
            \delta = \abs{\det\left(\mathop{\left(\frac{\partial g_i}{\partial T_j}(P_R)
                    \right)_{i=1\ldots N-d}}_{\hspace*{56pt} j=d+1\ldots N}\right)}_v =1.
        \end{equation*}
        Then the restrictions $T_1,\ldots,T_d$ to $\mathcal{Y}$ of the first $d$ coordinates
        form a system of local parameters of $\mathcal{Y}$ at $P_R$.  Let $\mathcal{Y}^{\prime}$
        be the union of all components that pass through $P_R$ of the scheme defined by
        the equations $g_1=\cdots=g_{N-d}=0$.  Since $\delta \neq 0$, the
        dimension of the tangent space $\Theta^{\prime}$ to
        $\mathcal{Y}^{\prime}$ at $P_R$ is $d$.  Since $\dim_P \mathcal{Y}^{\prime} \geq d$ and $\dim
        \Theta^{\prime} \geq \dim_P \mathcal{Y}^{\prime}$, we have that $\dim \mathcal{Y}^{\prime} = d$
        and $P_R$ is a non-singular point of $\mathcal{Y}^{\prime}$.  The scheme $\mathcal{Y}^{\prime}$ is
        irreducible and reduced, so then $\mathcal{Y}^{\prime} = \mathcal{Y}$.
        We have that $\mathcal{Y}$ is defined locally by $N-d$ equations (i.e.,
        $\mathcal{Y}$ is a local complete intersection)
        and those equations satisfy $\delta =1$.

        Since the ring of formal power series is complete, we apply
        the Implicit Function Theorem (Proposition \ref{thm3}) to deduce that there exists a system of power series $F_1,\ldots,F_{N-d}$ over $R$ in $d$ variables $T_1,\ldots,T_d$ such that
        $F_i(T_1,\ldots,T_d)$ converges for all $\abs{T_j}_v < 1=\delta^2$ and
        \begin{equation} \label{eq3}
            g_i(T_1,\ldots,T_d,F_1(T_1,\ldots,T_d),\ldots,F_{N-d}(T_1,\ldots,T_d))=0
        \end{equation}
        for $1\leq i \leq N-d$.
        Moreover, the coefficients of the power series $F_i$ are uniquely determined by
        (\ref{eq3}).
        Let $\tau: \O_{\mathcal{Y},P_R} \to K[[T]]$ be the uniquely determined map that
        takes each
        function to its Taylor series.  Assuming that $T_1,\ldots,T_d$ are chosen as local
        parameters, the formal power series
        \begin{equation*}
            \tau(T_{d+1}),\ldots,\tau(T_{N})
        \end{equation*}
        also satisfy (\ref{eq3}), and, hence, must coincide with $F_1,\ldots,F_{N-d}$.
        It follows that
        \begin{equation*}
            \tau(T_{d+1}),\ldots,\tau(T_{N})
        \end{equation*}
        converge for
        \begin{equation*}
            \abs{T_i}_v < 1=\delta^2, \quad \text{for } 1 \leq i \leq d.
        \end{equation*}
        We have $P_R=(0,\ldots,0)$ in $\A^{N}_R$ and
        $\overline{\phi_R(P_R)}=(\overline{0},\ldots,\overline{0})$ in
        $\A^{N}_k$.  Hence, we have
        \begin{equation*}
        \overline{\phi_R^{s}(P_R)} = (\overline{0},\ldots,\overline{0})\quad  \text{for all } s \geq 0.
        \end{equation*}
        Equivalently,
        \begin{equation*}
            \phi^s_R(P_R) \in \m \quad \text{for all } s \geq 0.
        \end{equation*}
        So we have that $\mathcal{P}$ is mapped to the affine neighborhood where $\O_{\mathcal{Y},P_R}$ has its maximal ideal generated by $T_1,\ldots,T_d$.  Note that we also have that the generators of the maximal ideal of $\O_{\mathcal{V},P_R}$ is a subset of $\{T_1,\ldots,T_d\}$. Choose a neighborhood $\mathcal{U} \subset \mathcal{Y}$ such that $\phi_R(P_R) \in \mathcal{U}$.

        The morphism $\phi_R$ can be written locally as a vector of homogeneous polynomials of the same degree
        $D$ with coefficients in $R$ and at least one coefficient in $R^{\ast}$.  Denote
        $\overline{\phi}$ the restriction of $\phi_R$ to the special fiber.  Recall that
        $[1,0,\ldots,0]$ is a fixed point of the morphism $\overline{\phi}$; hence
        \begin{equation} \label{eq45}
            \overline{\phi([1,0,\ldots,0])} = \overline{\phi}(\overline{[1,0,\ldots,0]})=\overline{[1,0,\ldots,0]}.
        \end{equation}
        Therefore, there is some normalized representation $[\Phi_0,\ldots,\Phi_N]$ of $\phi_R$ near $[1,0\ldots,0]$
        for which $\overline{\phi} = [\overline{\Phi_0},\ldots,\overline{\Phi_N}]$ is a morphism.
        Label such a representation as
        \begin{equation*}
            \Phi_i = \sum_{\abs{I}=D} \alpha_{i,I} \textbf{T}^I,
        \end{equation*}
        where $I$ is a multi-index.  By (\ref{eq45}), we have
        \begin{equation} \label{eq46}
                \overline{[\alpha_{0,(D,0,\ldots,0)},\ldots,\alpha_{N,(D,0,\ldots,0)}]} =
                \overline{[1,0\ldots,0]}.
        \end{equation}
        We will now check that $\alpha_{i,(D,0,\ldots,0)} \in R^{\ast}$ for some $0 \leq i \leq
        N$.
        Assume that
        \begin{equation*}
            \alpha_{i,(D,0,\ldots,0)} \equiv 0 \mod \pi \quad \text{for }0 \leq i \leq N.
        \end{equation*}
        Since $\overline{\phi}(\overline{[1,0,\ldots,0])}=\overline{[1,0,\ldots,0]}$ and by assumption
        every monomial contains one of $T_1,\ldots,T_N$, we must have
        some monomial in $T_1,\ldots,T_N$ dividing every $\overline{\Phi_i}$ for $0 \leq i \leq N$.
        This contradicts the fact that $\overline{\phi}$ is a morphism represented by
        $[\overline{\Phi_0},\ldots,\overline{\Phi_N}]$ at $\overline{[1,0,\ldots,0]}$.
        Consequently, by (\ref{eq46}), we have
        \begin{align*}
            \alpha_{i,(D,0,\ldots,0)} &\equiv 0 \mod{\pi} \quad \text{for} \quad 1 \leq i \leq N \\
            \alpha_{0,(D,0,\ldots,0)} &\not\equiv 0 \mod{\pi} \quad \text{i.e.,} \quad
            \alpha_{0,(D,0,\ldots,0)}\in R^{\ast}.
        \end{align*}
        On the affine neighborhood $\mathcal{U}$, we have that $T_1,\ldots,T_d$ form a system of local parameters of $\O_{\mathcal{Y},P}$ so
        we can dehomogenize and write
        \begin{equation*}
            \left(\frac{\Phi_1}{\Phi_0}(\rho),\ldots,\frac{\Phi_d}{\Phi_0}(\rho)\right)=\textbf{f}(\rho)=
            (f_1(\rho),\ldots,f_d(\rho)) = \phi(\rho)  \quad \text{for all} \quad \rho \in
            \mathcal{U},
        \end{equation*}
        yielding each $f_i$ as the quotient of polynomials with coefficients in $R$:
        \begin{equation*}
            f_i = \frac{\sum_{\abs{I} \leq D}\alpha_{i,I}\textbf{T}^I}
            {\sum_{\abs{I} \leq D}\alpha_{0,I}\textbf{T}^I} \quad \text{for} \quad 1
            \leq i \leq d.
        \end{equation*}
        Since $\alpha_{0,(D,0,\ldots,0)} \in R^{\ast}$, we can divide the numerators and denominators by
        $\alpha_{0,(D,0,\ldots,0)}^{-1}$ to get
        \begin{equation*}
            f_i = \frac{\sum_{\abs{I} \leq D}\alpha^{\prime}_{i,I}\textbf{T}^I}
            {1 + \sum_{\abs{I} \leq D}\alpha^{\prime}_{0,I}\textbf{T}^I}.
        \end{equation*}
        The denominators are series (in fact polynomials) with coefficients in $R$ and,
        hence, converge $\pi$-adically for $T_i \in \m$ and $0 \leq i \leq d$, so we can find a multiplicative inverse.  We can write
        \begin{equation*}
            f_i = \omega_i + \Lambda_i \textbf{T} + \cdots
        \end{equation*}
        where $\omega_i$ and $\Lambda_i$ are vectors.  These power series converge for values in
        $\m$ and have coefficients in $R$, and every iterate of $P_R$ has coordinates
        in $\m$.  Hence, we can iterate $\textbf{f}$ at $P_R$.

        We now have near $P_R$ on $\mathcal{Y}$ that
        \begin{equation*}
            \phi_R = \textbf{f} = \omega  + d\textbf{f}_P \textbf{T} + (\text{higher degree
            terms}),
        \end{equation*}
        where $d\textbf{f}_P$ is the Jacobian matrix of $\textbf{f}$ at $P_R$.

        To get the representation of $\phi_R |_{\mathcal{V}}$ we project $\textbf{f}$ onto $\O_{\mathcal{V},P_R} \subset \O_{\mathcal{Y},P_R}$ by sending $T_{d'+1} = \cdots= T_d =0$.  This gives us a representation $\textbf{f}\;'$ for $\phi_R|_{\mathcal{V}}$ with the appropriate properties.
    \end{proof}

    \begin{proof}[Proof of Theorem \ref{thm_intro_mrp}]
        Recall that we know $m \mid n$ from Corollary \ref{cor1}.  If $m=n$ we are done; otherwise, replace $\phi_R$ by $\phi_R^{m}$ and $n$ by
        $n/m$, so $m=1$. From Lemma \ref{lem17}, after dehomogenizing, we have $P_R=(0,\ldots,0)$ and on $\mathcal{V}$ we have
        \begin{equation*}
            \phi_R|_{\mathcal{V}} = \textbf{f} = \omega  + d\textbf{f}_P \textbf{T} + (\text{higher degree
            terms}),
        \end{equation*}
        where $d\textbf{f}_P$ is the Jacobian of $\textbf{f}$ at $P_R$, and we can iterate
        $\textbf{f}$ at $P_R$.  The morphism $\phi_R$ acts on $\mathcal{V}$ as a cyclic group of order $n$, thus $(\phi_R^n)_{|\mathcal{V}}$ is the identity map; in other words, $d\phi^n_P$ fixes $T_1,\ldots,T_{d^{\prime}}$.  Let $r_V$ be the order of $d\overline{\phi}_P$ on $V$ and note that $r_V \mid n$ since $\phi_R^n$ fixes $\mathcal{V}$.  Replace $n$ by $n/r_V$, $\phi_R$ by $\phi_R^{r_V}$, and $P_R$ by $\phi_R^{r_V}(P_R)$.  We will continue to write
        \begin{equation*}
            \phi_R|_{\mathcal{V}} = \textbf{f} = \omega  + d\textbf{f}_P \textbf{T} + (\text{higher degree
            terms}).
        \end{equation*}

        If $n=1$ we are done, so assume $n \neq 1$.  We know that $d\overline{\phi}_P$ fixes $T_1,\ldots,T_{d^{\prime}}$.  In other words, $d\textbf{f}_P$ is the identity modulo $\pi$.  Labeling $df_P = I + \pi^cA$ where $I$ is the identity matrix and $A$ has at least one entry in $R^{\ast}$.  If $df_P = I$, then set $c=\infty$.  Iterating modulo $(\omega)^2$ we have
        \begin{equation*}
            \textbf{f}^{\,n}(0) \equiv \left(nI + \sum_{i=1}^n\binom{i}{1}\pi^{c}A + \cdots + \sum_{i=1}^n \binom{i}{n}\pi^{nc}A^n\right)\omega
            \equiv 0 \mod{(\omega)^2}.
        \end{equation*}
        Since $\omega_j$ generates $(\omega)$ for some $1 \leq j \leq d'$ we must have
        \begin{equation*}
            n \equiv 0 \pmod{\pi}.
        \end{equation*}
        Replace $n$ by $n/p$ and $\phi_R$ by $\phi_R^p$.  If $n=1$ we are done; if not we can repeat the above argument to see that
        \begin{equation*}
            n = mr_Vp^e \quad \text{for} \quad e \geq 0.
        \end{equation*}

        Note that $r_V$ is independent of the choice of local parameters and power series representation for $\phi_R$ near $P_R$.

        We have that $d\overline{\phi^m}_P$ restricted to $\mathcal{V}$ is an element of $\GL_{d'}(\F_{N\pi})$; consequently, its multiplicative order is bounded by $((N\pi)^{d'} - 1)$ \cite[Corollary 2]{Darafsheh}.
    \end{proof}

    \begin{rem}
        We can use Theorem \ref{thm_intro_mrp} to bound the primitive period of a periodic point using good reduction information as in \cite[Corollary B]{Silverman10}, but such a result is not stated here since the proof is identical to the one-dimensional case.  Furthermore, these types of bounds, using only information at the primes of good reduction, have been superseded in the one-dimensional case by Benedetto \cite[Main Theorem]{Benedetto2}.
    \end{rem}

    \begin{thm} \label{thm36}
        Let $K$ be a number field and $X \subseteq \P^{N}_K$ a projective variety
        defined over $K$.  Let $\phi:X \to X$ be a morphism defined over $K$ and $P \in X(K)$ a
        periodic point of primitive period $n$.  For a prime of good reduction $\mathfrak{p} \in K$, let $m_{\mathfrak{p}}$ be the primitive period of $\overline{P}$.
        Then there are only finitely many primes $\mathfrak{p} \in K$ where $n \neq m_{\mathfrak{p}}$.
    \end{thm}

    \begin{proof}
        Let $[t_0,\ldots, t_N]$ be coordinates for $\P^{N}_K$.  Denote the $i\tth$ coordinate
        of a point $Q\in \P^{N}_K$ as $Q_i$.  Then we know that, after
        normalizing the points so that the first non-zero coordinate is 1, we have
        \begin{equation*}
            \phi^n(P)_i - P_i =0
        \end{equation*}
        for each coordinate $0 \leq i \leq N$.  We know that there are only
        finitely many divisors of $n \in \Z$, so we have finitely many normalized
        expressions
        \begin{equation*}
            \phi^s(P)_i -P_i
        \end{equation*}
        with $s \mid n$ and $s < n$.  We have
        \begin{equation*}
            \phi^s(P)_i - P_i \neq 0 \quad \text{for } 0 \leq i \leq N
        \end{equation*}
        since $n$ is the primitive period of $P$.  Hence, $\phi^s(P)_i - P_i$ is some value in $K$, which can be factored into finitely many primes.  We can do this for each of the finitely
        many $0 \leq i \leq N$. So there are at most finitely many primes where there is some
        $s \mid n$ with $s < n$ such that $\overline{\phi^{s}(P)} = \overline{P}$.  Hence, there
        are only finitely many primes $\mathfrak{p}$ such that $n \neq m_{\mathfrak{p}}$.
    \end{proof}

\subsection{Proof of Theorem \ref{thm_intro_pe}} \label{sect4}
    We now bound the exponent $e$ in Theorem \ref{thm_intro_mrp}.  We replace $\phi_R$ by $\phi_R^{mr_V}$ and $P_R$ by $\phi_R^{mr_V}(P_R)$.  From Lemma \ref{lem17}, we can reduce to the case of considering a local power series representation of $\phi_R |_{\mathcal{V}}$ at the origin.    We write
    \begin{equation*}
        \phi_R|_{\mathcal{V}} = \textbf{f} = \omega + d\textbf{f}_0\textbf{T} + \text{(higher degree terms)}
    \end{equation*}
    where $d\textbf{f}_0$ is the identity modulo $\pi$.  We will use the notation $x_t(P)$ to denote the $t\tth$ coordinate of a point $P$.  The method of proof is similar to \cite{Pezda}.

    \begin{defn}
        \mbox{}
        \begin{itemize}
            \item Define $F_k(x_1,\ldots,x_n)=\textbf{f}^{\,p^k}(x_1,\ldots,x_n)$.
            \item Define $\textbf{x}_0 = (0,\ldots,0)$ and $\textbf{x}_k = \textbf{f}^{\,p^k}(\textbf{x}_0) =F_k(\textbf{x}_0)= \textbf{f}^{\,p}(\textbf{x}_{k-1}) = F_{k-1}^p(\textbf{x}_0)$.
            \item Define $b_k = \min_{i=1,\ldots, d'}v(x_i(\textbf{x}_k))$.
            \item Write $d\textbf{f}^{\,p^k}_0 = I + \pi^{c_k}A_k$ where $I$ is the identity matrix and $A_k$ has at least one entry in $R^{\ast}$.  If $d\textbf{f}^{\,p^k}_0 = I$, then we set $c_k = \infty$.
            \item Assume that $(0,\ldots,0)$ has primitive period $p^e$.
        \end{itemize}
    \end{defn}

    \begin{lem} \label{lem5}
        For $k=1,\ldots, e-1$ we have
         \begin{align*}
            b_{k+1} &\geq  \min(2b_k,b_k + (p-1)c_k,b_k + v(p)) \\
            c_{k+1} &\geq \min(b_k,c_k + v(p),pc_k).
        \end{align*}
        Furthermore,
        \begin{equation*}
            \min(b_k,c_k) \leq v(p)
        \end{equation*}
        and if $b_k > v(p)$, then $(p-1)c_k \leq v(p)$.
    \end{lem}
    \begin{proof}
        We have
        \begin{align}
           \textbf{x}_{k+1} &\equiv F_k^p(\textbf{x}_0) \equiv  \left(I + (I + \pi^{c_k}A_k) +  \cdots + (I + \pi^{c_k}A_k)^{p-1}\right)\textbf{x}_k \pmod{\pi^{2b_k}}  \notag \\
           &\equiv \Bigg(pI + \sum_{i=1}^{p-1} \binom{i}{1}\pi^{c_k}A_k + \cdots \label{eq1}\\
           &\qquad \cdots + \sum_{i=1}^{p-1} \binom{i}{p-1}\pi^{(p-1)c_k}A_k^{p-1}\Bigg)\textbf{x}_k \pmod{\pi^{2b_k}}. \notag
        \end{align}
        Using the identity
        \begin{equation}\label{eq2}
            \sum_{j=1}^n \binom{j}{k} = \frac{n+1}{n-k+1}\binom{n}{k}
        \end{equation}
        we see that each intermediary sum in (\ref{eq1}) is divisible by $p$.  So we have
        \begin{equation*}
            b_{k+1} \geq \min(2b_k,b_k + (p-1)c_k,b_k + v(p)).
        \end{equation*}

        By the chain rule, we have
        \begin{align*}
            df^{k+1}_0 &\equiv (df^k_0)^p \equiv  (I + \pi^{c_k}A_k)^p \pmod{\pi^{b_k}}\\
            &\equiv I + \binom{p}{1}\pi^{c_k}A_k + \binom{p}{2}\pi^{2c_k}A_k^2 + \cdots + \pi^{pc_k}A_k^{p} \pmod{\pi^{b_k}}.
        \end{align*}
        Therefore,
        \begin{equation*}
            c_{k+1} \geq \min(b_k,c_k + v(p), pc_k).
        \end{equation*}

        To prove the final statement consider (\ref{eq1}) with $k=e-1$:
        \begin{equation*}
            \textbf{x}_{e} \equiv \left(pI +  \cdots
            + \pi^{(p-1)c_{e-1}}A_{e-1}^{p-1}\right)\textbf{x}_{e-1} \pmod{\pi^{2b_{e-1}}}.
        \end{equation*}
        By (\ref{eq2}) each intermediary term is divisible by $p\pi^{c_{e-1}}$.  If $b_{e-1} > v(p)$, then we must have $(p-1)c_{e-1} \leq v(p)$.  If the assertion fails for some $k$, then $\min(b_{e-1},c_{e-1}) > v(p)$.
    \end{proof}

    \begin{lem} \label{lem2}
        For $k = 1,\ldots, e-1$ and $p \neq 2$ we have
        \begin{align*}
            b_k &\geq 2^k \quad \text{for } b_k \leq v(p)\\
            c_k &\geq 2^{k-1}
        \end{align*}
        and for $p=2$ we have
        \begin{align*}
            b_k &\geq G_{k+1}\quad \text{for } b_k \leq v(p)\\
            c_k &\geq G_{k}
        \end{align*}
        where $\alpha = \frac{1+\sqrt{5}}{2}$ and $G_k = \frac{\alpha^{k} - (-1/\alpha)^k}{\sqrt{5}}$ is the $k\tth$ Fibonacci number.
    \end{lem}
    \begin{proof}
        Consider first $p \neq 2$.  Since $b_1 \geq 1$ and $c_1 \geq 1$ we have established $k=1$.  Assume $b_k \geq 2^k$ and $c_k \geq 2^{k-1}$ for some $k$.  If $b_k \leq v(p)$, then from Lemma \ref{lem5} we have $b_{k+1} \geq 2^{k+1}$ and $c_{k+1} \geq 2^k$.  If $b_k > v(p)$, then from Lemma \ref{lem5} we have $c_{k+1} \geq p c_k \geq 3 \cdot 2^{k-1}$, establishing the lemma for $p \neq 2$.

        If $p=2$, we have $b_1 \geq 1$ and $c_1 \geq 1$ establishing $k=1$.  So assume $b_k \geq G_{k+1}$ and $c_k \geq G_{k}$ for some $k$.  If $b_k \leq v(p)$, then by Lemma \ref{lem5} we have $b_{k+1} \geq G_{k+2}$ and $c_{k+1} \geq G_{k+1}$ since $G_{i+1} \leq 2\cdot G_{i}$ for each $i > 1$.  Similarly, if $b_k > v(p)$, then $c_{k+1} \geq 2c_k \geq G_{k+1}$.
    \end{proof}

    \begin{proof}[Proof of Theorem \ref{thm_intro_pe}]
        From Lemma \ref{lem5} we have $\min(c_{e-1},b_{e-1}) \leq v(p)$, and if $b_{e-1} > v(p)$, then $(p-1)c_{e-1} \leq v(p)$.  Using Lemma \ref{lem2}, for $p \neq 2$ the exponent $e$ must satisfy
        \begin{equation*}
            2^{e-1} \leq v(p)
        \end{equation*}
        and for $p=2$ the exponent $e$ must satisfy
        \begin{equation*}
            F_{e-1}  = \frac{\alpha^{e-1} - (-1/\alpha)^{e-1}}{\sqrt{5}} \leq v(2).
        \end{equation*}
        The formulas stated are now clear.
    \end{proof}

    \begin{proof}[Proof of Corollary \ref{cor_intro}]
        Apply the bounds from Theorem \ref{thm_intro_mrp} and Theorem \ref{thm_intro_pe} with $v(p)=1$.  To bound $m$ we note that there are at most $p$ choices for each coordinate of $\overline{P}$, so we have $m \leq p^d$.
    \end{proof}

    \section{Proof of Theorem \ref{thm_etale}}
        \begin{proof}
            Since $\phi$ is \'etale it sends smooth points to smooth points and permutes the irreducible components of $X$.  Therefore, there exists a finite integer $l$ such that $\phi^l:Y \to Y$.  Since $Y$ is smooth and irreducible and the restriction of $\phi^l$ to $Y$ has good reduction, we may apply Theorem \ref{thm_intro_mrp} and Theorem \ref{thm_intro_pe}.
        \end{proof}

\bibliography{masterlist}
\bibliographystyle{plain}

\end{document}